\documentclass[12pt,a4paper]{amsart}

\usepackage{tikz-cd} 
\tikzcdset{row sep/normal=3.6em, column
	sep/normal=3.6em}

\usepackage{mathrsfs} 
\usepackage{enumitem} 
\setlist{leftmargin=*}
\usepackage{caption}

\usepackage{stmaryrd}
\usepackage{showlabels}

\newtheorem{theorem}{Theorem}[section]
\newtheorem{proposition}[theorem]{Proposition}
\newtheorem{lemma}[theorem]{Lemma}

\theoremstyle{definition}
\newtheorem{assumption}[theorem]{Assumption}
\newtheorem{construction}[theorem]{Construction}
\newtheorem{convention}[theorem]{Convention}
\newtheorem{corollary}[theorem]{Corollary}
\newtheorem{definition}[theorem]{Definition}
\newtheorem{example}[theorem]{Example}
\newtheorem{examples}[theorem]{Examples}

\newtheorem{notation}[theorem]{Notation}
\newtheorem{openproblem}[theorem]{Open Problem}
\newtheorem{remark}[theorem]{Remark}

\usepackage{newmacros}

\begin{document}

\title[Quantitative and Continuous Algebras]{Varieties of Quantitative or Continuous Algebras (Extended Abstract)}
\author{J.~Adámek}
\thanks{J.~Adámek and M.~Dostál acknowledge the support by
	the Grant Agency of the Czech Republic under the grant 22-02964S}
\address{Department of Mathematics, Faculty of Electrical Engineering, Czech Technical University in Prague, Czech Republic and Institute for Theoretical Computer Science, Technical University Braunschweig, Germany}
\email{j.adamek@tu-bs.de}
\author{M.~Dostál}
\author{J.~Velebil}
\address{Department of Mathematics, Faculty of Electrical Engineering, Czech Technical University in Prague, Czech Republic}
\email{$\{$dostamat,velebil$\}$@fel.cvut.cz}

\begin{abstract}
	Quantitative algebras are algebras enriched in the category $\Met$ of metric spaces so that all operations are nonexpanding.
	Mardare, Plotkin and Panangaden introduced varieties (aka $1$-basic varieties) as classes of quantitative algebras presented by quantitative equations.
	We prove that they bijectively correspond to strongly finitary monads $T$ on $\Met$.
	This means that $T$ is the Kan extension of its restriction to finite discrete spaces.
	An analogous result holds in the category $\CMet$ of complete metric spaces.
	
	Analogously, continuous algebras are algebras enriched in $\CPO$, the category of $\omega$-cpos, so that all operations are continuous.
	We introduce equations between extended terms, and prove that varieties (classes presented by such equations) correspond bijectively to strongly finitary monads $T$ on $\CPO$.
	This means that $T$ is the Kan extension of its restriction to finite discrete cpos.
	(The two results have substantially different proofs.)
	An analogous result is also presented for monads on $\DCPO$.
	
	We also characterize strong finitarity in all the categories above by preservations of certain weighted colimits.
	As a byproduct we prove that directed colimits commute with finite products in all cartesian closed categories.
\end{abstract}

\maketitle

\section{Introduction}

Quantitative algebraic reasoning was formalized in a series of articles of Mardare, Panangaden and Plotkin \cite{bmpp18,mpp16,mpp17,bmpp21} as a tool for studying computational effects in probabilistic computation.
Those papers work with algebras in the category $\Met$ of metric spaces or $\CMet$ of complete metric spaces.
Metrics are always understood to be extended: the distance $\infty$ is allowed; morphisms are the nonexpanding maps $f$ which means that for $x,y$ in the domain one has $d(x,y) \leq d(f(x),f(y))$.
\emph{Quantitative algebras} are algebras acting on a (complete) metric space $A$ so that every $n$-ary operation is a nonexpanding map from $A^n$ (with the maximum metric) to $A$.
Mardare et al.\ introduced quantitative equations, which are formal expressions $t =_\eps t'$ where $t$ and $t'$ are terms and $\eps \geq 0$ is a rational number.
A quantitative algebra $A$ satisfies this equation iff for every interpretation of the variables the elements of $A$ corresponding to $t$ and $t'$ have distance at most $\eps$.
A \emph{variety} (called $1$-basic variety in~\cite{mpp16}) is a class of quantitative algebras presented by a set of quantitative equations.
Classical varieties of (non-structured) algebras are well known to correspond bijectively to \emph{finitary} monads $\Tmon = (T,\mu,\eta)$ on $\Set$, i.e.\ $T$ preserves directed colimits: every variety is isomorphic to the category $\Set^\Tmon$ of algebras for $\Tmon$, and vice versa.
The question whether an analogous correspondence holds for quantitative algebras has been posed in two presentations of LICS 22, see~\cite{adamek:varieties-quantitative-algebras} and~\cite{mv}.
We answer this by working with enriched (i.e.\ locally nonexpanding) monads on the category $\Met$ of metric spaces, and introducing weighted colimits called \emph{precongruences}.
We prove that varieties of quantitative algebras bijectively correspond to categories $\Met^\Tmon$ for \emph{strongly finitary monads} $\Tmon$ on $\Met$. And we characterize these monads as precisely those that preserve directed colimits and colimits of precongruences (Theorem~\ref{T:main-c}).
Analogously for strongly finitary monads on the category $\CMet$ (Theorem~\ref{T:main-cc}).

We also study closely related \emph{continuous algebras} which are algebras acting on a cpo (a poset with joins of $\omega$-chains) so that their operations are continuous.
Here we use equations $t = t'$ between \emph{extended terms} which allow not only the formation of composite terms $t = \sigma(t_0,\dots,t_{n-1})$ for $n$-ary operations $\sigma$, but also the formation of formal joins $t = \bigvee_{k \in \Nat} t_k$ for countable collections of terms.
A \emph{variety} of continuous algebras is a class presented by a set of such equations.
We again work with enriched (i.e.\ locally continuous) monads.
We prove that varieties of continuous algebras bijectively correspond to categories $\CPO^\Tmon$ for \emph{strongly finitary monads} on $\CPO$.
And we characterize these monads as precisely those that preserve directed colimits and reflexive coinserters (Theorem~\ref{T:main-cpo}).
The proof substantially uses that (unlike $\Met$ and $\CMet$) the category $\CPO$ is cartesian closed.
We prove that in every cartesian closed category directed colimits commute with finite products (Theorem~\ref{T:ccc}). 

\textbf{Related Work}
The main tool of Mardare et at.\ (\cite{mpp16,mpp17}) are \emph{$\omega$-basic equations}: for a finite metric space $M$ on the set of variebles of terms $t$ and $t'$ one wries $M \vdash t =_\eps t'$.
An algebra $A$ satisfies this equation if every nonexpanding interpretation $f : M \to A$ of the variebles the elements corresponding to $t$ and $t'$ have distance at most $\eps$.
A class of quantitative algebras presented by such equations is called an $\omega$-variety.
Unfortunately, the free-algebra monad of an $\omega$-variety need not be finitary (\cite{adamek:varieties-quantitative-algebras}, Example~4.1).
However, when the category $\Met$ is substituted by its full subcategory $\UMet$ of ultrametric spaces, then $\omega$-varieties were proved in~\cite{adamek:varieties-quantitative-algebras} to correspond bijectively to enriched monads preserving directed colimits of split monomorphisms and surjective morphisms.

Full proofs of the results presented in this extended abstract can be found in~\cite{adv:quantitative-algebras} and~\cite{adv:continuous-algebras}.

\section{Strongly Finitary Monads}

\begin{assumption}            
	Throughout our paper we work with categories and functors enriched over a symmetric monoidal closed category $(\V,\tensor,I)$.
	We recall these concepts shortly and introduce strongly finitary monads on $\V$, giving a characterization of them via preservation of certain colimits.
	Our leading examples of $\V$ are (complete) metric spaces and (complete) partially ordered sets.
\end{assumption}

\begin{definition}[\cite{borceux:second}, 6.12]
	A symmetric monoidal category is given by a category $\V$, a bifunctor $\tensor: \V \times \V \to \V$ and an object $I$.
	Moreover, natural isomorphisms are given expressing that $\tensor$ is commutative and associative, and has the unit $I$ (all up to coherent natural isomorphism).
	Finally, for every object $Y$ a right adjoint of the functor $\blank \tensor Y : \V \to \V$ is given.
	We denote it by $[Y,\blank]$ and denote the morphism corresponding to $f : X \tensor Y \to Z$ by $\wh{f} : Y \to [X,Z]$.
\end{definition}

Often $\tensor$ is the categorical product and $I$ the terminal object; then $\V$ is called cartesian closed.

\begin{examples}
	\label{E:smc}
	\phantom{phantom}
	\begin{enumerate}
		
		\item $\V = \Pos$, the category of posets, is cartesian closed, $[X,Y]$ is the poset of all monotone maps $f : X \to Y$ ordered pointwise.
		Here $\wh{f} = \curry f$ is the curried form of $f$.
		
		\item $\V = \CPO$, the category of cpos (more precisely: $\omega$-cpos) which are posets with joins of $\omega$-chains.
		Morphisms are the continuous maps: monotone maps preserving joins of $\omega$-chains.
		It is also cartesian closed, $[X,Y]$ is the cpo of all continuous maps (ordered again pointwise).
		Analogously $\V = \DCPO$ is the category of posets with directed joins (dcpos) where morphisms (also called continuous maps) preserve directed joins.
		
		\item $\V = \Met$, the category of (extended) metric spaces and nonexpanding maps.
		Objects are metric spaces defined as usual, except that the distance $\infty$ is allowed.
		Nonexpanding maps are those maps $f : X \to Y$ with $d(x,x') \geq d(f(x),f(x'))$ for all $x,x' \in X$.
		A product of metric spaces $A_i$ ($i \in I$) is the metric space on $\prod_{i \in I} A_i$ with the \emph{supremum metric}
		\[
		d((x_i),(y_i)) = \sup_{i \in I} d(x_i, y_i).
		\]
		This category is not cartesian closed: curryfication is not bijective.
		However, $\Met$ is symmetric closed monoidal w.r.t.\ the \emph{tensor product} $X \tensor Y$ which is the cartesian product with the \emph{addition metric}
		\[
		d((x,y), (x',y')) = d(x,x') + d(y,y').
		\]
		Here $[X,Y]$ is the metric space $\Met(X,Y)$ of all morphisms $f : X \to Y$ with the supremum metric: the distance of $f,g : X \to Y$ is $\sup_{x \in X} d(f(x),g(x))$.
		And $I$ is the singleton space.
		
		\item The category $\CMet$ of complete metric spaces is the full subcategory of $\Met$ on spaces with limits of all Cauchy sequences.
		It has the same symmetric closed monoidal structure as above: if $X$ and $Y$ are complete spaces, then so are $X \tensor Y$ and $[X,Y]$.
		
	\end{enumerate}
\end{examples}

\begin{notation}
	\label{N:K}
	\phantom{phantom}
	\begin{enumerate}
		
		\item Every set $X$ is considered as a \emph{discrete cpo} with $x \sqleq x'$ iff $x = x'$.
		This is the coproduct $\coprod_X I$ in $\CPO$ (and also in $\DCPO$).
		Analogously, $X$ is considered as a \emph{discrete metric space}: all distances of $x \neq x'$ are $\infty$.
		This is the coproduct $\coprod_X I$ in $\Met$ (and also in $\CMet$).
		
		\item For the category $\Set_\fin$ of finite sets and mappings we define a functor
		\[
		K : \Set_\fin \to \V, \qquad X \mapsto \coprod_X I.
		\]
		Thus for $\V = \Met$, $\CMet$, $\CPO$ or $\DCPO$ it assigns to every finite set the corresponding discrete metric space or discrete cpo, respectively.
		
	\end{enumerate}
\end{notation}

\begin{convention}
	By a \emph{catgory} $\C$ we always mean a category enriched over $\V$.
	It is given by
	\begin{enumerate}
		
		\item a class $\ob \C$ of objects,
		
		\item an object $\C(X,Y)$ of $\V$ (called the hom-object) for every pair $X,Y$ in $\ob \C$, 
		
		\item a 'unit' morphism $u_X : I \to \C(X,X)$ in $\V$ for every object $X \in \ob \C$, and
		
		\item 'composition' morphisms
		\[
		c_{X,Y,Z} : \C(X,Y) \tensor \C(Y,Z) \to \C(X,Z)
		\]
		for all $X,Y,Z \in \ob \C$, subject to commutative diagrams expressing the associativity of composition and the fact that $u_X$ are units of composition.
		For details see~\cite{borceux:second},~6.2.1.
		
	\end{enumerate}
\end{convention}

\begin{examples}
	\phantom{phantom}
	\begin{enumerate}
		
		\item If $\V = \Met$ then $\C$ is an ordinary category in which every hom-set $\C(X,Y)$ carries a metric such that composition is nonexpanding.
		Analogously for $\V = \CMet$.
		
		\item If $\V = \CPO$ then each hom-set $\C(X,Y)$ carries a cpo such that composition is continuous.
		Analogously for $\DCPO$.
		
	\end{enumerate}
\end{examples}

Let us recall the concept of an \emph{enriched functor} $F : \C \to \C'$ for (enriched) categories $\C$ and $\C'$.
It assigns
\begin{enumerate}
	
	\item an object $FX \in \ob \C'$ to every object $X \in \ob \C$, and
	
	\item a morphism $F_{X,Y} : \C(X,Y) \to \C'(FX,FY)$ of $\V$ to every pair $X,Y \in \ob \C$ so that the expected diagrams expressing that $F$ preserves composition and identity morphisms commute.
	
\end{enumerate}

\begin{convention}
	By a functor we always mean an enriched functor.
	We use 'ordinary functor' in the few cases where a non-enriched functor is meant.
\end{convention}

\begin{examples}
	\phantom{phantom}
	\begin{enumerate}
		
		\item For categories enriched over $\Met$ a functor $F : \C \to \C'$ is an ordinary functor which is \emph{locally nonexpanding}: given $f,g \in \C(X,Y)$ we have $d(f,g) \geq d(Ff,Fg)$.
		Analogously for $\CMet$.
		
		\item For categories enriched over $\Pos$ functors $F$ are the locally monotone ordinary functors: given $f \sqleq g$ in $\C(X,Y)$, we get $Ff \sqleq Fg$ in $\C(FX,FY)$.
		
		\item If $\V = \CPO$, then $F$ is an ordinary functor which is \emph{locally continuous}: it is locally monotone and for all $\omega$-chains $f_n : X \to Y$ in $\C(X,Y)$ we have $F(\bigsqcup_{n < \omega} f_n) = \bigsqcup_{n < \omega} Ff_n$.
		Analogously for $\DCPO$.
		
	\end{enumerate}
\end{examples}

\begin{remark}
	\phantom{phantom}
	\begin{enumerate}
		
		\item In general one also needs the concept of an enriched natural transformation between parallel (enriched) functors.
		However, if $\V$ is one of the categories of Example~\ref{E:smc}, this concept is just that of an ordinary natural transformation between the underlying ordinary functors.
		
		\item Given two categories $\D,\C$, we denote by $[\D,\C]$ the category of all functors $F: \D \to \C$ enriched by assigning to every pair of functors $F,G : \D \to \C$ an appropriate object $[F,G]$ of $\V$ of all natural transformations.
		
		In case $\V = \Pos$, $\CPO$ or $\DCPO$ the order of $[F,G]$ is component-wise: given $\tau,\tau' : F \to G$ put $\tau \sqleq \tau'$ iff $\tau_X \sqleq \tau_X'$ holds in $[FX,GX]$ of all $X \in \ob \D$.
		
		In case $\V = \Met$ or $\CMet$ the distance of $\tau,\tau'$ is $\sup_{X \in \ob \D} d(\tau_X,\tau_X')$.
		
	\end{enumerate}
\end{remark}

\begin{definition}[\cite{borceux:second,kelly:book}]
	\label{D:wc}
	A \emph{weighted diagram} in a category $\C$ is given by a functor $D : \D \to \C$ together with a weight $W: \D^\op \to \V$.
	A \emph{weighted colimit} is an object $C = \Colim{W}{D}$ of $\C$ together with isomorphisms  in $\V$:
	\[
	\psi_X : \C(C,X) \to [\D^\op,\C](W,\C(D\blank,X))
	\]
	natural in $X \in \ob \C$.
	The \emph{unit} of this colimit is the natural transformation 
	\[
	\nu = \psi_C(\id_C) : W \to \C(D\blank, C).
	\]
\end{definition}

In all categories of Example~\ref{E:smc} weighted colimits (for all $\D$ small) exist.

\begin{example}
	\label{E:conical}
	\emph{(Conical) directed colimits} are the special case where $\D$ is a directed poset and the weight $W$ is trivial: the constant functor with value $\One$ (the terminal object).
	
	\begin{enumerate}
		
		\item In $\Pos$ directed colimits are formed on the level of the underlying sets.
		They commute with finite products.
		
		\item In $\CPO$ directed colimits exist, but are not formed on the level of the underlying sets.
		For example, the finite ordinals $A_n = \{ 0, \dots, n-1 \}$ form a directed diagram with inclusions $A_n \hookrightarrow A_{n+1}$ as connecting maps for $n < \omega$.
		The colimit of this diagram in $\CPO$ is
		\[
		\Nat^\top = \colim_{n < \omega} A_n,
		\]
		the chain of natural numbers with a top element $\top$ added.
		Still, directed colimits commute with finite products in $\CPO$:
	\end{enumerate}
\end{example}

\begin{theorem}
	\label{T:ccc}
	In every cartesian closed category $\C$ directed colimits commute with finite products.
\end{theorem}

\begin{proof}
	\phantom{phantom}
	\begin{enumerate}
		
		\item Suppose $D,D' : \D \to \C$ are diagrams, where $\D$ is a directed poset.
		Given colimit cocones $c_d : Dd \to C$ and $c_d' : D'd \to C$, it is our task to prove that for the diagram $D \times D' : \D \to \C$ (given by $d \mapsto Dd \times D'd$) the cocone $c_d \times c_d' : Dd \times D'd' \to C \times C'$ is a colimit, too.

		\item Define a diagram $D * D' : \D \times \D \to \C$ by $(d,d') \mapsto Dd \times D'd'$.
		We shall prove that it has the colimit $c_d \times c_{d'}' : Dd \times D'd' \to C \times C'$.
		This proves the theorem: since $\D$ is directed, the diagonal $\diag: \D\to \D \times \D$ is a cofinal functor, thus $D * D'$ has the same colimit as $D \times D' = (D * D') \comp \diag$.
		
		\item Given a cocone $f_{d,d'} : Dd \times D'd' \to E$ of $D * D'$, we prove that it factorizes through $c_d \times c_{d'}'$; it is easy to verify that the factorization is unique.
		Fix an object $d' \in \D$ and form the adjoint transposes $\wh{f}_{d,d'}: D'd' \to [Dd,E]$ for all $d \in \D$.
		They form a cocone of $D$, thus, there exists a unique factorization through the cocone $c_d$. That is, we have a unique $g_{d'}: D'd' \times C \to E$ with $\wh{f}_{d,d'} = \wh{g}_{d'} \comp c_d$ (for all $d \in \D$).
		For the isomorphism $u: C \times D'd' \to D'd' \times C$ put $h_{d'} = g_{d'} \comp u : C \times D'd' \to E$ and form adjoint transposes $\wh{h}_{d'} : D'd' \to [C,E]$ for all $d' \in \D$.
		This is a cocone of $D'$, thus there exists a unique factorization through the cocone $c_{d'}'$: we have a unique $h: C \times C' \to E$ with $h_{d'}' = \wh{h} \comp c_{d'}'$ (for all $d' \in \D$).
		It follows that $h$ is the desired factorization of $f_{d,d'}$:
		\[
		h \comp (c_d \times c_{d'}') = h \comp (C \times c_{d'}') \comp (c_d \times C') = h_{d'} \comp (c_d \times D'd') = f_{d.d'}.
		\]
		
	\end{enumerate}
\end{proof}

\begin{example}
	\label{E:dir-met}
	(Conical) directed colimits in $\Met$ and $\CMet$ also exist.
	Again, they are not formed on the level of the underlying sets.
	For example, consider the diagram of metric space $A_n = \{ 0, 1 \}$ with $d_n(0,1) = 2^{-n}$, where the connecting maps are $\id: A_n \to A_{n+1}$ ($n < \omega$).
	The colimit is a singleton space.
\end{example}

\begin{theorem}
	Directed colimits in $\Met$ or $\CMet$ commute with finite products.
\end{theorem}

\begin{proof}[Proof sketch]
	\phantom{phantom}
	\begin{enumerate}
		
		\item For directed diagrams $(D_i)_{i \in I}$ in $\Met$, cocones $c_i : D_i \to C$ forming a colimit were characterized in~\cite{adamek+rosicky:approximate-injectivity}, Lemma~2.4, by the following properties:
		\begin{enumerate}[label=(\alph*)]
			
			\item $C = \bigcup_{i \in I} c_i[D_i]$, and
			
			\item for every $i \in I$, given $y,y' \in D_i$ we have
			\[
			d(c_i(y),c_i(y')) = \inf_{j \geq i} d(f_j(y), f_j(y'))
			\]
			where $f_j : D_i \to D_j$ denotes the connecting map.
			
		\end{enumerate}
		Given another directed diagram $(D_i')_{i \in I}$ with a cocone $c_i' : D_i' \to C'$ satisfying (a) and (b), it is our task to prove that the cocone $c_i \times c_i' : D_i \times D_i' \to C \times C'$ satisfies (a), (b), too.
		Since $I$ is directed, (a) is clear, and (b) needs just a short computation.
		
		\item For directed colimits in $\CMet$ the characterization of colimit cocones is analogous: (b) is unchanged, and in (a) one states that $\bigcup_{i \in I} c_i[D_i]$ is dense in $C$.
		The further argument is then analogous to (1).
		
	\end{enumerate}
\end{proof}

\begin{example}
	\label{E:coins}
	\phantom{phantom}
	\begin{enumerate}
		
		\item For our next development an important type of a weighted colimit in $\Pos$, $\CPO$ or $\DCPO$ is the \emph{coinserter}.
		Let $f_0,f_1 : A \to B$ be an ordered parallel pair.
		Its coinserter is a universal morphism $c : B \to C$ w.r.t.\ the property $c \comp f_0 \sqleq c \comp f_1$.
		Universality means that
		\begin{enumerate}
			\item every morphism $c'$ with $c' \comp f_0 \sqleq c' \comp f_1$ factorizes through $c$ and
			\item given $u,v : C \to D$ with $u \comp c \sqleq v \comp c$, it follows that $u \sqleq v$.
		\end{enumerate}
		For $\Pos$ this is precisely the weighted colimit of the diagram $D: \D \to \Pos$ where $\D$ consists of a single parallel pair $\delta_0, \delta_1 : d \to \ol{d}$ (where $\D(d,\ol{d})$ is a discrete poset) and $D \delta_i = f_i$.
		The weight $W : \D^\op \to \Pos$ is given by $W\ol{d} = \{ 0, 1\}$ where $0 < 1$, $Wd = \{ * \}$ and $W \delta_i(*) = i$.
		Analogously for $\CPO$ or $\DCPO$.
		
		\item A concrete example: every poset $C$ is a coinserter of a parallel pair between discrete posets.
		Indeed, let $|C|$ be the discrete poset underlying $C$ and let $C^{(2)} \subseteq |C| \times |C|$ be the set of all pairs $x_0 \sqleq x_1$ in $C$.
		For the pair of projections $\pi_0,\pi_1 : C^{(2)} \to |C|$ the coinserter is $\id: |C| \to C$.
		Analogously for $\CPO$ and $\DCPO$.
		
	\end{enumerate}
\end{example}

\begin{definition}
	A coinserter $c$ of $f_0,f_1$ is called \emph{surjective} if $c$ is a surjective map.
	It is called \emph{reflexive} if $f_0,f_1$ is a reflexive pair, i.e.\ they are split epimorphisms with a joint splitting $d$ ($f_0 \comp d = f_1 \comp d = \id$).
\end{definition}

\begin{example}
	The coinserter of Example~\ref{E:coins}~(2) is reflexive and surjective.
	In $\Pos$, all coinserters are surjective, in $\CPO$ they are not in general.
\end{example}

\begin{proposition}
	\label{P:ccc}
	In $\Pos$, $\CPO$ and $\DCPO$ reflexive coinserters commute with finite products.
\end{proposition}

The proof is similar to that of Theorem~\ref{T:ccc}.

Analogously to coinserters of discrete cpos yielding all cpos, we now introduce weighted colimits in~$\Met$ of diagrams called precongruences (a name borrowed from~\cite{bourke-garner:monads-and-theories}). They express every metric space as a colimit of discrete spaces. (The weight used for precongruence is, however, not discrete.)

In the following definition $|M|$ denotes the underlying set (a discrete metric space) of a metric space $M$.

\begin{definition}
	\phantom{phantom}
	\begin{enumerate}
		
		\item We define the \emph{basic weight} $W_0 : \D_0^\op \to \Met$ as follows.
		The category $\D_0$ has an object $a$ and objects $\eps$ for every rational number $\eps > 0$.
		The only non-trivial hom-spaces are the spaces
		\[
		\D_0(\eps,a) = \{ \lambda_\eps, \rho_\eps \} \text{ with } d(\lambda_\eps,\rho_\eps) = \eps
		\]
		Thus $\D_0$ consists of the discrete category of positive rationals together with a pair of cocones (having codomain $a$).
		The values of $W_0$ are $W_0 a = \{ 0 \}$ and $W_0 \eps = \{ l, r \}$ with $d(l,r) = \eps$.
		The morphisms $W_0 \lambda_\eps, W_0 \rho_\eps: \{ 0 \} \to \{ l, r \}$ are given by $0 \mapsto l$, $0 \mapsto r$, respectively.
		
		\item For every metric space $M$ we define its \emph{precongruence} as the weighted diagram $D_M : \D_0 \to \Met$ with the basic weight $W_0$, where $D_M a = |M|$ and $D_M \eps \subseteq |M| \times |M|$ is the set of all pairs of distance at most $\eps$.
		Here $D \lambda_\eps, D \rho_\eps : D_M \eps \to |M|$ are the left and right projections, respectively.
		
	\end{enumerate}
\end{definition}

\begin{proposition}
	\label{P:prec}
	Every metric space $M$ is the weighted colimit of its precongruence in $\Met$.
\end{proposition}

\begin{proof}
	Given a space $X$, to give a natural transformation $\tau: W_0 \to [D_M\blank,X]$ means to specify a map $f = \tau_a(0) : |M| \to X$ together with $\tau_\eps(l), \tau_\eps(r) : D_M \eps \to X$ with $\tau_\eps(l) = f \comp \pi_l$ and $\tau_\eps(r) = f \comp \pi_r$.
	Thus $\tau$ is determined by $f$ and the given equations are equivalent to $f : M \to X$ being nonexpanding.
	The desired isomorphism $\psi_X$ of Definition~\ref{D:wc} is given by $\psi_X(\tau) = f$.
\end{proof}

\begin{remark}
	Analogously we define precongruences in $\CMet$: we just use the codomain restrictions $W_0 : \D_0^\op \to \CMet$ and $D_M : \D_0 \to \CMet$.
	Again, every complete space is the weighted colimit of its precongruence in $\CMet$.
\end{remark}

\begin{remark}
	\phantom{phantom}
	\begin{enumerate}
		
		\item Let $C = \Colim{W}{D}$ be a weighted colimit of $D : \D \to \C$ with unit $\nu : W \to \C(D\blank, C)$.
		Given an (enriched) functor $F : \C \to \C'$, it \emph{preserves the colimit} provided that the diagram $FD : \D \to \C'$ with weight $W$ has the colimit $\Colim{W}{FD} = FC$ with the unit $\ol{\nu} : W \to \C'(FD\blank,FC)$
		having components $\ol{\nu}_d = F\nu_d$.
		
		\item A functor is \emph{finitary} if it preserves directed colimits.
		
	\end{enumerate}
\end{remark}

\begin{example}
	\phantom{phantom}
	\label{E:power}
	\begin{enumerate}
		
		\item The endofunctor $(\blank)^n$ of $\Met$ or $\CMet$ preserves colimits of precongruences for every $n \in \Nat$.
		This is easy to verify.
		By Example~\ref{E:dir-met} $(\blank)^n$ is finitary.
		
		\item The endofunctor $(\blank)^n$ on $\CPO$ or $\Pos$ preserves reflexive coinserters for every $n \in \Nat$ by Proposition~\ref{P:ccc}, and is finitary by Theorem~\ref{T:ccc}.
		
	\end{enumerate}
\end{example}

Let us recall the concept of the (enriched) \emph{left Kan extension}~\cite{kelly:book} of a functor $F : \A \to \C$ along a functor $K : \A \to \C$: this is an endofunctor $\Lan{K}{F} : \C \to \C$ endowed with a universal natural transformation $\tau: F \to (\Lan{K}{F}) \comp K$.
The universal property states that given a natural transformation $\sigma : F \to G \comp K$ for any endofunctor $G : \C \to \C$, there exists a unique natural transformation $\ol{\sigma} : \Lan{K}{F} \to G$ with $\sigma = \ol{\sigma}K \comp \tau$.

\begin{definition}
	An endofunctor $F$ of $\V$ is \emph{strongly finitary} if it is a left Kan extension of its restriction $F \comp K$ to $\Set_\fin$:
	\[
	F = \Lan{K}{(F \comp K)}
	\]
	(see Notation~\ref{N:K}).
\end{definition}

\begin{examples}
	\phantom{phantom}
	\begin{enumerate}
		
		\item An endofunctor of $\Set$ is strongly finitary iff it is finitary.
		
		\item An endofunctor of $\Pos$ is strongly finitary iff it is finitary and preserves reflexive coinserters, see~\cite{adv:ordered-algebras}.
		
	\end{enumerate}
\end{examples}

In order to characterize strong finitarity for $\V = \CPO$, $\DCPO$, $\Met$ and $\CMet$, we apply Kelly's concept of density presentation that we now recall.

\begin{notation}
	Let $K : \A\to \C$ be a functor.
	We denote by $\wt{K} : \C \to [\A^\op, \V]$ the functor $\wt{K}C = \C(K\blank, C)$.
\end{notation}

For example, the functor $K : \Set_\fin \to \CPO$ yields $\wt{K} : \CPO \to [\Set_\fin^\op,\CPO]$ taking a cpo $C$ to the functor $C^{(\blank)} : \Set_\fin^\op \to \CPO$ of finite powers of $C$.
Analogously for $K : \Set_\fin \to \Met$.

\begin{definition}[\cite{kelly:book}]
	A \emph{density presentation} of a functor $K: \A \to \C$ is a collection of weighted colimits in $\C$ such that
	\begin{enumerate}[label=(\arabic*)]
		
		\item $\wt{K}$ preserves those colimits, and
		
		\item $\C$ is the (iterated) closure of the image $K[\A]$ under those colimits.
	\end{enumerate}
\end{definition}

\begin{example}
	\label{E:density}
	\phantom{phantom}
	\begin{enumerate}[label=(\alph*)]
		
		\item $K : \Set_\fin \to \Pos$ has a density presentation consisting of directed colimits and reflexive coinserters.
		Indeed, Condition~(2) in the above definition follows from Example~\ref{E:coins}~(2): every finite poset is a coinserter of a pair in $\Set_\fin$.
		And every poset is a directed colimit of finite subposets.
		For Condition~(1) observe that $\wt{K} : \Pos \to [\Set_\fin^\op, \Pos]$ assigns to every poset $A$ the functor of its finite powers $A^{(\blank)} : \Set_\fin^\op \to \Pos$.
		This functor preserves directed colimits (Proposition~\ref{T:ccc}) and reflexive coinserters (Example~\ref{E:power}).
		
		\item $K : \Set_\fin \to \CPO$ also has a density presentation consisting of directed colimits and reflexive coinserters.
		Here Condition~(1) is verified as above, using Proposition~\ref{T:ccc} and Example~\ref{E:power}.
		To verify Condition~(2), we express every cpo in four steps, starting from $\Set_\fin$:
		\begin{enumerate}[label=(\roman*)]
			
			\item Every finite cpo is a reflexive coinserter of a parallel pair in $\Set_\fin$ (Example~\ref{E:coins}~(2)).
			
			\item The cpo $\Nat^\top$ is a directed colimit of finite cpos $A_n$ (Example~\ref{E:conical}~(2)).
			Analogously, every copower $r \tens \Nat^\top$ of $r$ copies, $r \in \Nat$, is a directed colimit of $r \tens A_n$ for $n < \omega$.
			
			\item In this step we create all reflexive coinserters $C$ of pairs $f_0,f_1 : r \tens \Nat^\top \to r' \tens \Nat^\top$ for all $r,r' \in \Nat$.
			Such cpos $C$ are called \emph{basic}.
			
			\item The proof is concluded by proving that every cpo $A$ is a directed colimit of the diagram of all of its basic sub-cpos $A_i$ ($i \in I$).
			In fact, that this diagram is directed follows from the fact that a coproduct of two basic cpos is clearly basic.
			Given a cocone $s_i : A_i \to S$, we are to prove that there is a unique continuous map $s : A \to S$ extending each $s_i$.
			For each $x \in A$ the subposet $\{ x \}$ is clearly basic.
			Thus given $i$ with $x \in A_i$ the value $s_i(x)$ is independent of $i$.
			This follows easily from the compatibility of the cocone $s_i$.
			The desired map is defined by $s(x) = s_i(x)$.
			The verification that $s$ is continuous is a bit more subtle.
			
		\end{enumerate}
	\end{enumerate}
\end{example}

\begin{remark}
	\label{R:dense-c}
	\phantom{phantom}
	\begin{enumerate}
		
		\item The reflexive coinserters used in steps (i) and (iii) of the last example are all surjective.
		Consequently, $K : \Set_\fin \to \CPO$ also has the density presentation consisting of directed colimits and reflexive, surjective coinserters.
		
		\item The functor $K : \Set_\fin \to \DCPO$ also has the density presentation of all directed colimits and reflexive (surjective) coinserters.
		The proof is the same as for $\CPO$: all cpos used in the last example are indeed dcpos.
		
	\end{enumerate}
	
\end{remark}

\begin{corollary}
	\label{C:density}
	An endofunctor of $\CPO$ or $\DCPO$ is strongly finitary iff it preserves directed colimits and reflexive (surjective) coinserters.
\end{corollary}

Indeed, by~\cite{kelly:book}, Theorem~5.29, given a density presentation of $K : \Set_\fin \to \CPO$, strong finitarity means precisely preservation of all colimits in that presentation.
The same corollary also holds in $\DCPO$.

\begin{example}
	The categories $\Met$ and $\CMet$ have a density presentation of $K$ consisting of all directed diagrams and precongruences of finite spaces.
	Indeed, Condition~(1) follows from Examples~\ref{E:dir-met} and~\ref{E:power}~(1).
	For Condition~(2) observe that finite metric spaces are obtained from $\Set_\fin$ as colimits of precongruences by Lemma~\ref{P:prec}, and every (complete) metric space is a directed colimit of all of its finite subspaces in $\Met$ (or $\CMet$, resp.).
\end{example}

\begin{corollary}
	\label{C:coco}
	An endofunctor of $\Met$ or $\CMet$ is strongly finitary iff it preserves directed colimits and colimits of precongruences.
\end{corollary}

\begin{example}
	\label{E:sf}
	Every coproduct of endofunctors $(\blank)^n$ with $n$ finite on $\Met$, $\CMet$, $\CPO$ or $\DCPO$ is strongly finitary.
	Indeed, strongly finitary functors are closed under coproducts, which follows directly from the definition.
\end{example}

\section{Varieties of Quantitative Algebras}

We now prove that varieties of quantitative algebras on the categories $\Met$ and $\CMet$ bijectively correspond to \emph{strongly finitary monads}.
These are monads carried by a strongly finitary endofunctor.
Throughout this section $\Sigma = (\Sigma_n)_{n \in \Nat}$ denotes a signature, and $V$ is a specified countable set of variables.

\begin{notation}
	\phantom{phantom}
	\begin{enumerate}
		
		\item Following Mardare, Panangaden and Plotkin~\cite{mpp16}, a \emph{quantitative algebra} is a metric space $A$ endowed with a nonexpanding operation $\sigma_A : A^n \to A$ for every $\sigma \in \Sigma_n$ (w.r.t.\ the supremum metric (Example~\ref{E:smc})).
		We denote by $\Sigma\text{-}\Met$ the category of quantitative algebras and nonexpanding homomorphisms.
		Its forgetful functor is denoted by $U_\Sigma : \Sigma\text{-}\Met \to \Met$.
		
		\item Analogously, a \emph{complete quantitative algebra} is a quantitative algebra carried by a complete metric space.
		The category $\Sigma\text{-}\CMet$ is the corresponding full subcategory of $\Sigma\text{-}\Met$.
		We again use $U_\Sigma : \Sigma\text{-}\CMet \to \CMet$ for the forgetful functor.
		
		\item The underlying set of a metric space $M$ is denoted by $|M|$.
		
	\end{enumerate}
\end{notation}

\begin{example}
	\label{E:free-q}
	\begin{enumerate}
		
		\item A free quantitative algebra on a metric space $M$ is the usual algebra $T_\Sigma M$ of \emph{terms} on variables from $M$.
		That is, the smallest set containing $|M|$ and such that for every $n$-ary symbol $\sigma$ and every $n$-tuple of terms $t_i$ ($i < n$) we obtain a composite term $\sigma(t_i)_{i<n}$.
		To describe the metric, let us introduce the equivalence $\sim$ on $T_\Sigma M$ (\emph{similarity} of terms): it is the smallest equivalence making all variables of $|M|$ into one class, and such that $\sigma(t_i)_{i < n} \sim \sigma'(t_i')_{i < n}$ holds iff $\sigma = \sigma'$ and $t_i \sim t_i'$ for all $i < n$.
		The metric of $T_\Sigma M$ extends that of $M$ as follows: $d(t,t') = \infty$ if $t$ is not similar to $t'$.
		For similar terms $t = \sigma(t_i)$ and $t' = \sigma(t_i')$ we put $d(t,t') = \sup_{i < n} d(t_i,t_i')$.
		
		\item If $M$ is a complete space, $T_\Sigma M$ is also complete, and this is the free quantitative algebra on $M$ in $\Sigma\text{-}\CMet$.
		
	\end{enumerate}
\end{example}

In particular, consider the specified set $V$ of variables as a discrete metric space, then $T_\Sigma V$ is the discrete algebra of usual terms.
For every algebra $A$ and every interpretation of variables $f : V \to A$ (in $\Met$ or $\CMet$) we denote by $f^\sharp : T_\Sigma V \to A$ the corresponding homomorphism: it interprets terms in $A$.

\begin{definition}[\cite{mpp16}]
	By a \emph{quantitative equation} (aka $1$-basic quantitative equation) is meant a formal expression $t =_\eps t'$ where $t,t'$ are terms in $T_\Sigma V$ and $\eps \geq 0$ is a rational number.
	An algebra $A$ in $\Sigma\text{-}\Met$ (or $\Sigma\text{-}\CMet$) \emph{satisfies} that equation if for every interpretation $f : V \to A$ we have $d(f^\sharp(t), f^\sharp(t')) \leq \eps$.
	We write $t = t'$ in case $\eps = 0$.
	
	By a \emph{variety}, aka $1$-basic variety, of quantitative (or complete quantitative) algebras is meant a full subcategory of $\Sigma\text{-}\Met$ (or $\Sigma\text{-}\CMet$, resp.) specified by a set of quantitative equations.
\end{definition}

\begin{example}
	\phantom{phantom}
	\label{E:quant}
	\begin{enumerate}
		
		\item \emph{Quantitative monoids} are given by the usual signature: a binary symbol $\cdot$ and a constant $e$, and by the usual equations: $(x \cdot y) \cdot z = x \cdot (y \cdot z)$, $e \cdot x = x$, and $x \cdot e = x$.
		
		\item \emph{Almost commutative monoids} are quantitative monoids in which the distance of $ab$ and $ba$ is always at most $1$.
		They are presented by the quantitative equation $x \cdot y =_1 y \cdot x$.
	\end{enumerate}
\end{example}

\begin{proposition}[See~\cite{mpp16}]
	Every variety $\Vvar$ of quantitative algebras has free algebras: the forgetful funtor $U_\Vvar : \Vvar \to \Met$ has a left adjoint $F_\Vvar : \Met \to \Vvar$.
\end{proposition}

\begin{notation}
	We denote by $\Tmon_\Vvar$ the free-algebra monad of a variety $\Vvar$ on $\Met$.
	Its underlying functor is $T_\Vvar = U_\Vvar \comp F_\Vvar$.
\end{notation}

\begin{example}
	\label{E:ex}
	For $\Vvar = \Sigma\text{-}\Met$ we have seen the monad $T_\Sigma$ in Example~\ref{E:free-q}: $T_\Sigma M$ is the metric space of all terms over $M$.
	Observe that $T_\Sigma$ is a coproduct of endofunctors $(\blank)^n$, one summand for each similarity class of terms on $n$ variables over $M$ (which is independent of the choice $M$).
	Thus $\Tmon_\Sigma$ is a strongly finitary monad: see Example~\ref{E:sf}.
\end{example}

\begin{remark}
	\label{R:beck}
	\phantom{phantom}
	\begin{enumerate}
		
		\item Recall the comparison functor $K_\Vvar : \Vvar \to \Met^{\Tmon_\Vvar}$: it assigns to every algebra $A$ of $\Vvar$ the algebra on $U_\Vvar A$ for $\Tmon_\Vvar$ given by the unique homomorphism $\alpha: F_\Vvar U_\Vvar A \to A$ extending $\id_{U_\Vvar A}$.
		More precisely: $K_\Vvar A = (U_\Vvar A, U_\Vvar \alpha)$.
		
		\item By a \emph{concrete category} over $\Met$ is meant a category together with a faithful 'forgetful' functor $U_\Vvar : \Vvar \to \Met$.
		For example a variety, or $\Met^\Tmon$ for every monad $\Tmon$.
		A \emph{concrete functor} is a functor $F : \Vvar \to \Wvar$ with $U_\Vvar = U_\Wvar F$.
		For example, the comparison functor $K_\Vvar$.
		
	\end{enumerate}
\end{remark}

\begin{proposition}
	\label{P:beck}
	Every variety $\Vvar$ of quantitative algebras is concretely isomorphic to the category $\Met^{\Tmon_\Vvar}$: the comparison functor $K_\Vvar : \Vvar \to \Met^{\Tmon_\Vvar}$ is a concrete isomorphism.
\end{proposition}

\begin{proof}
	For classical varieties (over $\Set$) this is proved in~\cite{maclane:cwm}, Theorem~VI.8.1.
	The proof for $\Met$ in place of $\Set$ is analogous.
\end{proof}

\begin{notation}
	\label{N:subst}
	
	\begin{enumerate}
		
		\item Given a natural number $n$ denote by $[n]$ the signature of one $n$-ary symbol $\delta$.
		If a term $t \in T_\Sigma V$ contains at most $n$ variables (say, $x_0,\dots,x_{n-1}$), we obtain a monad morphism $\ol{t} : \Tmon_{[n]} \to \Tmon_\Sigma$ as follows.
		For every space $M$ the function $\ol{t}_M$ takes a term $s$ using the single symbol $\delta$ and substitutes each occurence of $\delta$ by $t(x_0,\dots,x_{n-1})$. 
		
		\item Every metric space $A$ defines a monad $\spitze{A,A}$ on $\Met$ assigning to $X \in \Met$ the space $\spitze{A,A}X = [[X,A],A]$.
		More precisely: the functor $[\blank,A] : \Met \to \Met^\op$ is self-adjoint, and $\spitze{A,A}$ is the monad corresponding to that adjunction.
		
		\item Let $\Tmon$ be a monad on $\Met$ and $\alpha : TA \to A$ an algebra for it.
		We denote by $\wh{\alpha}_X : TX \to \spitze{A,A}X$ the morphism adjoint to the following composite
		\[
		[X,A] \tensor TX \xrightarrow{T(\blank) \tensor TX} [TX,TA] \tensor TX \xrightarrow{\ev} TA \xrightarrow{\alpha} A.
		\]
		
	\end{enumerate}
\end{notation}

\begin{theorem}[\cite{dubuc}]
	Given an algebra $\alpha: TA \to A$ for a monad $\Tmon$ on $\Met$, the morphisms $\wh{\alpha}_X$ above form a monad morphism $\wh{\alpha} : \Tmon \to \spitze{A,A}$.
	Every monad morphism from $\Tmon$ to $\spitze{A,A}$ has that form for a unique $\alpha$.
\end{theorem}

\begin{lemma}
	\label{L:equa}
	Let $A$ be a $\Sigma$-algebra expressed by $\alpha : T_\Sigma A \to A$ in $\Met^{\Tmon_\Sigma}$.
	It satisfies a quantitative equation $l =_\eps r$ iff the distance of $\wh{\alpha} \comp \ol{l}, \wh{\alpha} \comp \ol{r} : \Tmon_{[n]} \to \spitze{A,A}$ is at most $\eps$.
\end{lemma}

\begin{notation}
	The category of finitary monads on $\Met$ (and monad morphisms) is denoted by $\Mon_\fin(\Met)$.
	Its full subcategory of strongly finitary monads by $\Mon_\sf(\Met)$.
	We also denote by $\Mon_c(\Met)$ the category of monads preserving countably directed colimits.
\end{notation}

\begin{lemma}
	\label{L:cl}
	The category $\Mon_\fin(\Met)$ has weighted colimits, and $\Mon_\sf(\Met)$ is closed under them.
\end{lemma}

\begin{proof}[Proof sketch]
	\begin{enumerate}
		
		\item The category $\Mon_c(\Met)$ is locally countably presentable as an enriched category, thus it has weighted colimits, and
		
		\item both $\Mon_\fin(\Met)$ and $\Mon_\sf(\Met)$ are coreflective subcategories of $\Mon_c(\Met)$.
		The coreflection of a countably accessible monad $\Tmon$ in $\Mon_\sf(\Met)$ is given by the left Kan extension $\Lan{K}{(T \comp K)}$, analogously for $\Mon_\fin(\Met)$.
		
	\end{enumerate}
\end{proof}

\begin{lemma}
	\label{L:help2}
	Every monad morphism $\alpha: \Tmon_\Sigma \to \Smon$ in the category $\Mon_\fin(\Met)$ factorizes as a morphism $\Tmon_\Sigma \to \ol{\Smon}$ with surjective components followed by a morphism $\ol{\Smon} \to \Smon$ whose components are isometric embeddings.
\end{lemma}

\begin{theorem}
	\label{T:var-mon}
	For every variety $\Vvar$ of quantitative algebras the free-algebra monad $\Tmon_\Vvar$ is strongly finitary.
\end{theorem}

\begin{proof}[Proof sketch]
	\phantom{phantom}
	\begin{enumerate}
		
		\item Let $\Vvar$ be given by a signature $\Sigma$ and quantitative equations $l_i =_\eps r_i$ ($i \in I$), each containing $n_i$ variables.
		For every $i \in I$ we consider the signature $[n(i)]$ of one symbol $\delta_i$ of arity $n(i)$, then the terms $l_i,r_i$ yield the corresponding monad morphisms $\ol{l}_i, \ol{r}_i : \Tmon_{[n(i)]} \to \Tmon_\Sigma$ of Notation~\ref{N:subst}.
		An algebra $\alpha: T_\Sigma A \to A$ lies in $\Vvar$ iff the distance of $\wh{\alpha} \comp \ol{l}_i, \wh{\alpha} \comp \ol{r}_i : \Tmon_{[n(i)]} \to \spitze{A,A}$ is at most $\eps_i$ for each $i$ (Lemma~\ref{L:equa}).
		
		\item We verify that $\Tmon_\Vvar$ is a weighted colimit of strongly finitary monads in $\Mon_\fin(\Met)$.
		The domain $\D$ of the weighted diagram $D : \D \to \Mon_\fin(\Met)$ is the discrete category $I$ (indexing the equations) enlarged by a new object $a$, and by morphisms $\lambda_i,\rho_i : i \to a$ for every $i \in I$.
		Then put $D_i = \Tmon_{[n(i)]}$ and $Da = \Tmon_\Sigma$; further $D \lambda_i = \ol{l}_i$ and $D \rho_i = \ol{r}_i$.
		The weight $W : \D^\op \to \Met$ takes $i$ to the space $\{ l, r \}$ with $d(l,r) = \eps_i$ and $a$ to $\{ 0 \}$.
		We define $W\lambda_i(0) = l$ and $W \rho_i(0) = r$.
		The monads $\Tmon_\Sigma$ and $\Tmon_{[n(i)]}$ are strongly finitary by Example~\ref{E:ex}.
		This will finish the proof by Lemma~\ref{L:cl}.
		
		We denote by $\Tmon$ the weighted colimit $\Tmon = \Colim{W}{D}$ in $\Mon_\fin(\Met)$.
		The proof is concluded by proving that $\Vvar$ is isomorphic, as a concrete category, to the category $\Met^\Tmon$ of algebras for $\Tmon$.
		Then $\Tmon$ is the free-algebra monad of $\Vvar$.
		For $\Tmon$ we have the unit $\gamma: W \to \llbracket D \blank, \Tmon \rrbracket$ of the weighted colimit $\Tmon = \Colim{W}{D}$ (Definition~\ref{D:wc}).
		Its component $\nu_a$ assigns to $0$ a monad morphism $\gamma = \nu_a(0) : \Tmon_\Sigma \to \Tmon$, whereas for $i \in I$ the component $\nu_i$ is given by $l \mapsto \gamma \comp \ol{l}_i$ and $r \mapsto \gamma \comp \ol{r}_i$.
		Since $\nu_i$ is nonexpanding, we conclude that $\gamma \comp \ol{\lambda}_i, \gamma \comp \ol{\rho}_i: \Tmon_{[n(i)]} \to \Tmon$ have distance at most $\eps_i$.
		We thus obtain a functor $E : \Met^\Tmon \to \Vvar$ assigning to every algebra $\alpha : TA \to A$ the $\Sigma$-algebra corresponding to $\alpha \comp \gamma_A : T_\Sigma A \to A$: it satisfies $l_i =_{\eps_i} r_i$ due to $d(\gamma \comp \ol{\lambda}_i, \gamma \comp \ol{\rho}_i) \leq \eps_i$.
		Moreover, $\gamma$ has surjective components, which can be derived from Lemma~\ref{L:help2}.
		Therefore, $E$ is a concrete isomorphism, which concludes the proof.
		
	\end{enumerate}
\end{proof}

\begin{construction}
	\label{C:eq-q}
	In the reverse direction we assign to every strongly finitary monad $\Tmon = (T,\mu,\eta)$ on $\Met$ a variety $\Vvar_\Tmon$, and prove that $\Tmon$ is its free-algebra monad.
	For every morphism $k : X \to Y$ let us denote by $k^* = \mu_Y \comp Tk : TX \to TY$ the corresponding homomorphism in $\Met^\Tmon$.
	Recall our fixed set $V = \{  x_i \mid i \in \Nat \}$ of variables, and form, for each $n \in \Nat$, the finite discrete space $V_n = \{ x_i \mid i < n \}$.
	The signature we use has as $n$-ary symbols the elements of the space $TV_n$: $\Sigma_n = |TV_n|$ for $n \in \Nat$.
	The variety $\Vvar_\Tmon$ is given by the following quantitative equations, where each symbol $\sigma \in \Sigma_n$ is considered as the term $\sigma(x_1,\dots,x_{n-1})$ and $n,m$ below range over $\Nat$:
	\begin{enumerate}[left= 5pt]
		
		\item $\sigma =_\eps \sigma'$ for all $\sigma,\sigma' \in \Sigma_n$ with $d(\sigma,\sigma') \leq \eps$ in $TV_n$;
		
		\item $k^*(\sigma) = \sigma(k(x_i))_{i<n}$ for all $\sigma \in \Sigma_n$ and all maps $k : V_n \to \Sigma_m$;
		
		\item $\eta_{V_n}(x_i) = x_i$ for all $i = 0,\dots,n-1$.
		
	\end{enumerate}
\end{construction}

\begin{lemma}
	\label{L:homo}	
	Every algebra $\alpha : TA \to A$ in $\Met^\Tmon$ yields an algebra $A$ in $\Vvar_\Tmon$ with operations $\sigma_A : A^n \to A$ defined by $\sigma_A(a(x_i)) = a^*(\sigma)$ for all $\sigma \in \Sigma_n$ and $a : V_n \to A$.
	Moreover, every homomorphism in $\Met^\Tmon$ is also a $\Sigma$-homomorphism between the corresponding algebras in $\Vvar_\Tmon$.
\end{lemma}

\begin{proof}[Proof sketch]
	\begin{enumerate}
		
		\item The operation $\sigma_A$ is nonexpanding because $T$ is locally nonexpanding.
		It satisfies 1) in Construction~\ref{C:eq-q} because for every interpretation $a : V_n \to A$ we have $d(a^*(l),a^*(r)) \leq \eps$.
		Satisfaction of 2) follows from $a^* \comp k^* = (a^* \comp k)^*$, and 3) is clearly satisfied.
		Thus the $\Sigma$-algebra $A$ lies in $\Vvar_\Tmon$.
		
		\item Given a morphism $h : (A,\alpha) \to (B,\beta)$ in $\Met^\Tmon$ (i.e., $h \comp \alpha = \beta \comp Th$) we are to prove that $h \comp \sigma_A = \sigma_B \comp h^n$ for all $\sigma \in TV_n$.
		This follows easily from $h \comp a^* = (h \comp a)^*$ for each $a : V_n \to A$.
		
	\end{enumerate}
\end{proof}

\begin{theorem}
	\label{T:mon-var}
	Every strongly finitary monad $\Tmon$ on $\Met$ is the free-algebra monad of the variety $\Vvar_\Tmon$.
\end{theorem}

\begin{proof}
	For every metric space $M$ we want to prove that the $\Sigma$-algebra associated with $(TM,\mu_M)$ in Lemma~\ref{L:homo} is free in $\Vvar_\Tmon$ w.r.t.\ the universal map $\eta_M$.
	Then the theorem follows from Proposition~\ref{P:beck}.
	
	We have two strongly finitary monads, $\Tmon$ and the free algebra monad of $\Vvar_\Tmon$ (Theorem~\ref{T:var-mon}).
	Thus, it is sufficient to prove the above for finite discrete spaces $M$.
	Then this extends to all finite spaces because we have $M = \Colim{W_0}{D_M}$ (Lemma~\ref{P:prec}) and both monads preserve this colimit.
	Since they coincide on all finite discrete spaces, they coincide on $M$.
	Finally, the above extends to all spaces $M$: we have a directed colimit $M = \colim_{i \in I} M_i$ of the diagram of all finite subspaces $M_i$ ($i \in I$) which both monads preserve.
	
	Given a finite discrete space $M$, we can assume without loss of generality $M = V_n$ for some $n \in \Nat$.
	For every algebra $A$ in $\Vvar_\Tmon$ and an interpretation $f : V_n \to A$, we prove that there exists a unique $\Sigma$-homomorphism $\ol{f} : TV_n \to A$ with $f = \ol{f} \comp \eta_{V_n}$.
	\begin{description}
		\item[Existence] Define $\ol{f}(\sigma) = \sigma_A(f(x_i))_{i<n}$ for every $\sigma \in TV_n$.
		The equality $f = \ol{f} \comp \eta_{V_n}$ follows since $A$ satisfies the equations $\eta_{V_n}(x_i) = x_i$, thus the operation of $A$ corresponding to $\eta_{V_n}(x_i)$ is the $i$-th projection.
		The map $\ol{f}$ is nonexpanding: given $d(l,r) \leq \eps$, the algebra $A$ satisfies $l =_\eps r$.
		Therefore given an $n$-tuple $f : V_n \to A$ we have
		\[
		d(l_A(f(x_i)),r_A(f(x_i))) \leq \eps.
		\]
		To prove that $\ol{f}$ is a $\Sigma$-homomorphism, take an $m$-ary operation symbol $\tau \in TV_m$.
		We prove $\ol{f} \comp \tau_{V_m} = \tau_A \comp \ol{f}^m$.
		This means that every $k : V_m \to TV_n$ fulfils
		\[
		\ol{f} \comp \tau_{V_m} (k(x_j))_{j < m} = \tau_A \comp \ol{f}^m (k(x_j))_{j < m}.
		\]
		The definition of $\ol{f}$ yields that the right-hand side is $\tau_A(k(x_j)_A(f(x_i)))$.
		Due to equation (2) in Construction~\ref{C:eq-q} with $\tau$ in place of $\sigma$ this is $k^*(\tau)_A(f(x_i))$.
		The left-hand side yields the same result since
		\[
		\ol{f}^m(k(x_j)) = (k(x_j))_A (f(x_i)).
		\]
		\item[Uniqueness] Let $\ol{f}$ be a nonexpanding $\Sigma$-homomorphism with $f = \ol{f} \comp \eta_{V_n}$.
		In $TV_n$ the operation $\sigma$ asigns to $\eta_{V_n}(x_i)$ the value $\sigma$.
		(Indeed, for every $a : n \to |TV_n|$ we have $\sigma_{TV_n}(a_i) = a^*(\sigma) = \mu_{V_n} \comp Ta(\sigma)$.
		Thus $\sigma_{TV_n}(\eta_{V_n}(x_i)) = \mu_{V_n} \comp T \eta_{V_n} (\sigma) = \sigma^n$.)
		Since $\ol{f}$ is a homomorphism, we conclude
		\[
		f(\sigma) = \sigma_A(\ol{f} \comp \eta_{V_n} (x_i)) = \sigma_A(f(x_i))
		\]
		which is the above formula.
	\end{description}
\end{proof}

\begin{corollary}
	Varieties of quantitative algebras correspond bijectively, up to isomorphism, to strongly finitary monads on $\Met$.
\end{corollary}

Indeed, a stronger result can be deduced from Theorems~\ref{T:var-mon} and~\ref{T:mon-var}: let $\Var(\Met)$ denote the category of varieties of quantitative algebras and concrete functors (Remark~\ref{R:beck}).
Recall that $\Mon_\sf(\Met)$ denotes the category of strongly finitary monads.

\begin{theorem}
	\label{T:main-c}
	The category $\Var(\Met)$ of varieties of quantitative algebras is equivalent to the dual of the category $\Mon_\sf(\Met)$ of strongly finitary monads on $\Met$.
\end{theorem}

\begin{proof}
	Morphisms $\phi: \Smon \to \Tmon$ between monads in $\Mon_\sf(\Met)$ bijectively correspond to concrete functors $\ol{\phi} : \Met^\Tmon \to \Met^\Smon$ (\cite{barr+wells:toposes}, Theorem~3.3): $\ol{\phi}$ assigns to an algebra $\alpha: TA \to A$ of $\Met^\Tmon$ the algebra $\alpha \comp \phi_A : SA \to A$ in $\Met^\Smon$.
	We know that for every variety $\Vvar$ the comparison functor is invertible (Proposition~\ref{P:beck}).
	This yields a functor $\Phi: \Var(\Met)^\op \to \Mon_\sf(\Met)$ assigning to a variety $\Vvar$ the monad $\Tmon_\Vvar$ (Theorem~\ref{T:var-mon}).
	Given a concrete functor $F : \Vvar \to \Wvar$ between varieties, there is a unique monad morphism $\phi: \Tmon_\Wvar \to \Tmon_\Vvar$ such that $\ol{\phi} = K_\Wvar \comp F \comp K_\Vvar^{-1} : \Met^{\Tmon_\Vvar} \to \Met^{\Tmon_\Wvar}$.
	We define $\Phi F = \phi$ and get a functor which is clearly full and faithful.
	Thus Theorem~\ref{T:mon-var} implies that $\Phi$ is an equivalence of categories.
\end{proof}

The whole development of the present section works in $\Sigma\text{-}\CMet$ as well as in $\CMet$.
First observe that for every complete space $M$ the space $T_\Sigma M$ of Example~\ref{E:free-q} is also complete (being a coproduct of finite powers of $M$).
The resulting monad $\Tmon_\Sigma$ on the category $\CMet$ is strongly finitary (as in Example~\ref{E:ex}).
More generally, every variety $\Vvar$ of complete quantitative algebras yields a monad $\Tmon_\Vvar$ on $\CMet$ which is strongly finitary, and $\Vvar$ is isomorphic to $\Met^{\Tmon_\Vvar}$.
The proof is analogous to that of Theorem~\ref{T:var-mon}, just at the end we use, instead of the factorization system (surjective, isometric embedding) of $\Met$ the factorization system of $\CMet$ consisting of dense morphisms $f : A \to B$ (every element of $B$ is a limit of a sequence in $f[A]$) followed by closed isometric embeddings.
The proof that every strongly finitary monad on $\CMet$ is the free-algebra monad of a variety is completely analogous to that of Theorem~\ref{T:mon-var}.
We thus obtain

\begin{theorem}
	\label{T:main-cc}
	The category $\Var(\CMet)$ of varieties of complete quantitative algebras is equivalent to the dual of the category $\Mon_\sf(\CMet)$ of strongly finitary monads on $\CMet$.
\end{theorem}

\section{Varieties of Continuous Algebras}

For the categories $\Pos$, $\CPO$ and $\DCPO$ we obtain here and in Section~5 the same result: varieties of algebras bijectively correspond to strongly finitary monads.
For $\Pos$ we have proved this in~\cite{adv:ordered-algebras}.
The proof for $\CPO$ presented below is very different from the proofs in~\cite{adv:ordered-algebras} and in the previous section.
In fact, already the concept of equation is entirely different since it uses formal joins $\bigvee_{k \in \Nat} t_k$ of collections $t_0,t_1,t_2,\dots$ of terms.
The idea of such formal joins stems from~\cite{anr}, but our concept is slightly more restrictive: we request that all the terms $t_i$ contain only a finite set of variables.
We assume again that $\Sigma$ is a finitary signature, and that a countable set $V$ of variables has been chosen.
The underlying set of a cpo $M$ is denoted by $|M|$.

\begin{definition}
	A \emph{continuous algebra} is a cpo $A$ endowed with continuous operations $\sigma_A : A^n \to A$ for every $n$-ary symbol $\sigma \in \Sigma$ (w.r.t.\ the coordinate-wise order on $A^n$).
	We denote by $\Sigma\text{-}\CPO$ the category of continuous algebras and continuous homomorphisms.
\end{definition}

\begin{example}
	\label{E:free-c}
	A free continuous algebra on a cpo $M$ is the usual algebra $T_\Sigma M$ of terms on variables from $|M|$ (compare Example~\ref{E:free-q}) with the following order $\sqleq$ extending that of $M$: $t \sqleq t'$ iff $t$ and $t'$ are similar, $t = \sigma(t_i)_{i < n}$, $t' = \sigma(t_i')_{i < n}$ and such that $t_i \sqleq t_i'$ for every $i < n$.
\end{example}

In particular, considering $V$ as a discrete cpo (no distinct elements are comparable), then $T_\Sigma V$ is the discrete cpo of the usual terms.
For every continuous algebra $A$ and every interpretation $f : V \to A$ of variables we again denote by $f^\sharp : T_\Sigma V \to A$ the corresponding homomorphism.
As already mentioned, usual terms are not sufficient for equational presentations: we need formal joins of terms.
We use the symbol $\bigvee_{k \in \Nat}$ for them, while $\bigsqcup_{k \in \Nat}$ denotes $\omega$-joins in a given poset.

\begin{definition}
	\phantom{phantom}
	\begin{enumerate}
		
		\item The set $\T_\Sigma V$ of \emph{extended terms} is the smallest set containing $T_\Sigma V$, and such that for every countable collection $t_k$ ($k \in \Nat$) of extended terms containing only finitely many variables we get an extended term $\bigvee_{k \in \Nat} t_k$.
		
		\item For every continuous algebra $A$ and every interpretation $f : V \to A$ of variables we define the \emph{interpretations of extended terms} as the following partial function $f^@ : \T_\Sigma V \rightharpoonup A$:
		\begin{enumerate}[label=(\alph*)]
			\item $f^@$ extends $f^\sharp$ (thus $f^@(t)$ is defined for all terms $t \in T_\Sigma V$), and
			\item $f^@$ is defined in $t = \bigvee_{k \in \Nat} t_k$ iff each $f^@(t_k)$ is defined and fulfils $f^@(t_k) \sqleq f^@(t_{k+1})$ in $A$; then $f^@(t) = \bigsqcup_{k \in \Nat} f^@(t_k)$.
		\end{enumerate}
		
	\end{enumerate}
\end{definition}

\begin{example}
	Given a unary operation $\sigma$, the extended term $\bigvee_{k \in \Nat} \sigma^k(x)$ is well formed, but $\bigvee_{k \in \Nat} \sigma(x_k)$ is not: it contains infinitely many variables.
\end{example}

\begin{definition}
	\label{D:sat}
	By an \emph{equation} we understand a formal expression $t = t'$, where $t,t'$ are extended terms in $\T_\Sigma V$.
	
	A continuous algebra \emph{satisfies} $t = t'$  if for every interpretation $f : V \to A$ of the variables both $f^@(t)$ and $f^@(t')$ are defined and are equal.
\end{definition}

A \emph{variety} of continuous algebras is a full subcategory of $\Sigma\text{-}\CPO$ presented by a set of equations.

\begin{remark}
	\label{R:defi}
	\phantom{phantom}
	\begin{enumerate}
		
		\item We do not need presentation by inequations $t \leq t'$.
		Indeed, to satisfy such an inequation means precisely to satisfy $t = \bigvee_{k \in \Nat} t_k$ where $t_0 = t$ and $t_k = t'$ for all $k > 0$.
		
		\item A term $t$ is \emph{definable} in an algebra $A$ iff for every interpretation $f : V \to A$ of variables $f^@(t)$ is defined.
		Instead of equations, we can use definability to introduce varieties.
		Indeed, an algebra satisfies $t = t'$ iff the term $\bigvee_{k \in \Nat} s_k$ where $s_0 = t$, $s_1 = t'$ and $s_k = t$ for $k \geq 2$ is definable.
		Conversely, $t$ is definable in $A$ iff $A$ satisfies $t = t$.
		
	\end{enumerate}
\end{remark}

\begin{example}
	\phantom{phantom}
	\begin{enumerate}
		
		\item \emph{Continuous monoids} are monoids acting on cpos with continuous multiplication: for all $\omega$-chains $(a_k)$, $(b_k)$ we have $(\bigsqcup a_k)(\bigsqcup b_k) = \bigsqcup(a_k b_k)$.
		This is a variety presented by the usual monoid equations (see Example~\ref{E:quant}).
		
		\item Continuous monoids satisfying $a \sqleq a^2$ are presented by the definability of the term $\bigvee_{k \in \Nat} x^k$.
		
		\item The equation $\bigvee_{k \in \Nat} x^k = e$ presents continuous monoids satisfying $a \sqleq a^2$ and $\bigsqcup_{k \in \Nat} a^k = e$ for all $a$.
		
	\end{enumerate}
\end{example}

\begin{remark}
	\label{R:HSP}
	In the classical universal algebra Birkhoff's Variety Theorem states that varieties are precisely the HSP classes, i.e., closed under homomorphic images, subalgebras, and products.
	In $\Sigma\text{-}\CPO$ we have the corresponding constructions:
	\begin{enumerate}[label=(\roman*)]
		
		\item A \emph{product} of algebras $A_i$ ($i \in I$) is their cartesian product $\prod_{i \in I} A_i$ with operations and order given coordinate-wise.
		
		\item A \emph{subalgebra} of an algebra $A$ is a subobject $m : B \to A$ such that $m$ is an embedding: $x \sqleq y$ holds in $B$ iff $m(x) \sqleq m(y)$.
		
		\item A \emph{homomorphic image} of an algebra $A$ is a quotient object $e : A \to B$ such that $e$ is surjective.
		
	\end{enumerate}
\end{remark}

\begin{lemma}
	\label{L:HSP}
	Every variety of continuous algebras is an HSP-class.
\end{lemma}

The proof in~\cite{anr} on pp.~339-340 works in our setting without any changes.
The main result of~\cite{anr} is the converse implication, but the proof does not work for our extended terms (more special than the terms in op.~cit.):

\begin{openproblem}
	Is every HSP-class in $\Sigma\text{-}\CPO$ a variety of continuous algebras?
\end{openproblem}

\begin{definition}
	A subset $X$ of a cpo $C$ is \emph{dense} if the only sub-cpo containing $X$ is all of $C$.
\end{definition}

\begin{example}
	\label{E:dense}
	Given a directed diagram $D$ with a colimit $a_s : A_s \to A$ ($s \in S$) in $\CPO$, the union $\bigcup_{s \in S} a_s [A_s]$ is dense in $A$.
	Indeed, if a sub-cpo $A'$ contains that union, then the codomain restrictions $a_s' : A_s \to A'$ form a cocone of $D$ in $\CPO$.
	From the fact that this cocone factorizes through $a_s$ it follows that $A' = A$.
\end{example}

\begin{lemma}
	\label{L:dense}
	If $X$ is dense in a cpo $C$, then $X^n$ is dense in $C^n$ for each $n \in \Nat$.
\end{lemma}

\begin{proposition}
	\label{P:free}
	Every variety $\Vvar$ of continuous algebras has free algebras: the forgetful functor $U_\Vvar : \Vvar \to \CPO$ has a left adjoint $F_\Vvar : \CPO \to \Vvar$.
\end{proposition}

\begin{proof}[Proof sketch]
	For $\Vvar = \Sigma\text{-}\CPO$ we have described the free algebras in Example~\ref{E:free-c}.
	Every variety $\Vvar$ is closed under products and subalgebras by Lemma~\ref{L:HSP}.
	The category $\Sigma\text{-}\CPO$ is complete and wellpowered.
	It has the factorization system $(\EE,\MM)$ where $\EE$ consists of homomorphisms $e : A \to B$ with $e[A]$ dense in $B$ and $\MM$ consists of embeddings of closed subalgebras.
	(Indeed, let a continuous homomorphism $f : A \to C$ have such a factorization $f = m \comp e$ for $e : A \to X$ and $m : X \to C$ in $\Met$.
	Then there is a unique algebra structure on $X$ making $e$ and $m$ homomorphisms.)
	It follows that $\Vvar$ is a reflective subcategory: the inclusion functor $E : \Vvar \to \Sigma\text{-}\CPO$ has a left adjoint by Theorem~16.8 in~\cite{ahs}.
	Since $U_\Vvar = U_\Sigma \comp E$, we conclude that $U_\Vvar$ has a left adjoint.
\end{proof}

\begin{notation}
	\phantom{phantom}
	\begin{enumerate}
		
		\item We denote by $\Tmon_\Vvar$ the free-algebra monad of a variety $\Vvar$ on $\CPO$.
		Its underlying functor is $T_\Vvar = U_\Vvar \comp F_\Vvar$.
		
		\item Concrete categories (and functors) over $\CPO$ are defined analogously to Remark~\ref{R:beck}.
		
	\end{enumerate}
\end{notation}

\begin{example}
	For $\Vvar = \Sigma\text{-}\CPO$ the monad $\Tmon_\Sigma$ of Example~\ref{E:free-c} assigns to a cpo $M$ the cpo $T_\Sigma M$.
	As in Example~\ref{E:ex}, $\Tmon_\Sigma$ is strongly finitary.
\end{example}

\begin{proposition}
	Every variety $\Vvar$ of continuous algebras is concretely isomorphic to the category $\CPO^{\Tmon_\Vvar}$: the comparison functor $K_\Vvar : \Vvar \to \CPO^{\Tmon_\Vvar}$ is a concrete isomorphism.
\end{proposition}

This is, as Proposition~\ref{P:beck}, analogous to the classical case.

\begin{lemma}[\cite{anr}, Proposition~3.5]
	\label{L:help}
	Let $h : A \to B$ be a morphism in $\Sigma\text{-}\CPO$ and $t$ an extended term.
	Given an interpretation $f : V \to A$ with $f^@(t)$ defined, then $(hf)^@$ is also defined: $(hf)^@ (t) = h(f^@(t))$.
\end{lemma}

\begin{proposition}
	\label{P:create}
	The forgetful functor $U_\Sigma : \Sigma\text{-}\CPO \to \CPO$ creates directed colimits: given a directed diagram $D$ of continuous algebras with a colimit $c_i : UD_i \to C$ in $\CPO$, there exists a unique algebra structure on $C$ making all $c_i$ homomorphisms; moreover the resulting cocone $c_i : D_i \to C$ is a colimit of $D$ in $\Sigma\text{-}\CPO$.
\end{proposition}

\begin{proof}[Proof sketch]
	This follows from $\CPO$ being cartesian closed.
	Thus directed colimits commute with finite products (Theorem~\ref{T:ccc}).
	Given an $n$-ary symbol $\sigma \in \Sigma$, from the fact that $c_i^n : UD_i^n \to C^n$ is a directed colimit it follows that there is a unique morphism $\sigma_C : C_i^n \to C$ such that the given operations $\sigma_{C_i} : C_i^n \to C_i$ fulfil $\sigma_C \comp c_i^n = c_i \comp \sigma_{C_i}$.
	That is, $c_i$ are homomorphisms.
	The verification that this yields a colimit of $D$ in $\CPO$ is easy.
\end{proof}

\begin{proposition}
	\label{P:creates2}
	The functor $U_\Sigma$ creates reflexive coinserters.
\end{proposition}

\begin{proof}[Proof sketch]
	Indeed, if $e : B \to C$ is such a coinserter of $f_0,f_1 : A \to B$, then for an $n$-ary symbol $\sigma$ we have a coinserter $e^n$ of $f_0^n,f_1^n$ (Example~\ref{E:power}).
	We thus obtain a unique $\sigma_C : C^n \to C$ with $\sigma_C \comp e^n = e \comp \sigma_B$, hence $e$ is a homomorphism which is easily seen to be the coinserter in $\Sigma\text{-}\CPO$.
\end{proof}

\begin{theorem}
	\label{T:dir-var}
	Every variety of continuous algebras is closed under directed colimits in $\Sigma\text{-}\CPO$.
\end{theorem}

\begin{proof}[Proof sketch]
	Let $a_s : A_s \to A$ ($s \in S$) be a directed colimit in $\Sigma\text{-}\CPO$.
	Consider an extended term $t \in \T_\Sigma V$ such that in each of the algebras $A_s$ $t$ is definable.
	We prove that $t$ is definable in $A$.
	This concludes the proof using Remark~\ref{R:defi}.
	We use structural induction:
	the statement is obvious if $t$ is a classical term.
	Thus we need to prove the induction step: let $t = \bigvee_{k \in \Nat} t_k$ with each $t_k$ definable in all $A_s$ ($s \in S$), then $t$ is definable in $A$.
	By the definition of extended terms we know that there is a finite set $V_0 \subseteq V$ of variables with $t_k \in \T_\Sigma V_0$ for all $k \in \Nat$.
	Thus we can work with interpretations $f : V_0 \to A$.
	The colimit cocone $a_s : A_s \to A$ has the property that the set $X = \bigcup_{s \in S} a_s[A_s]$ is dense (Example~\ref{E:dense}).
	In the cpo $[V_0,A]$ of all interpretations the subset $[V_0,X]$ is thus dense, too (Lemma~\ref{L:dense}).
	We use this to prove that $f^@(t)$ is defined by structural induction on $f$:
	\begin{enumerate}[label=(\roman*)]
		
		\item if $f[V_0] \subseteq X$ then $f^@$ is defined, and
		
		\item given $f = \bigsqcup_{n \in \Nat} f_n$ in $[V_0,A]$ with all $f_n^@(t)$ defined, then $f^@(t)$ is defined.
		
	\end{enumerate}
	Step (i) is easy, using the fact that since $V_0$ is finite, $f : V_0 \to \bigcup a_s[A_s]$ factorizes through one of the subset $a_s[A_s]$, as $f = a_s \comp \ol{f}$.
	For the interpretaion $\ol{f} : V_0 \to A_s$ we know that $\ol{f}^@(t)$ is defined.
	Then we apply Lemma~\ref{L:help} to $h = a_s$.
	Step (ii) is more involved since it works with double induction: for $t_k$ and $f_n$.
\end{proof}

\begin{corollary}
	\label{C:strong}
	The monad $\Tmon_\Vvar$ on $\CPO$ is strongly finitary for every variety $\Vvar$ of continuous algebras.
\end{corollary}

\begin{proof}
	We know that $U_\Sigma : \Sigma\text{-}\CPO \to \CPO$ creates directed colimits and reflexive coinserters (Propositions~\ref{P:create} and~\ref{P:creates2}).
	
	Given a variety $\Vvar$, the embedding $E: \Vvar \hookrightarrow \Sigma\text{-}\CPO$ preserves directed colimits (Theorem~\ref{T:dir-var}) and reflexive surjective homomorphisms: indeed, $\Vvar$ is closed under homomorphic images (Lemma~\ref{L:HSP}).
	Consequently, the forgetful functor $U_\Vvar = U_\Sigma \comp E$ preserves directed colimits and reflexive surjective coinserters.
	Its left adjoint $F_\Vvar$ preserves weighted colimits.
	Thus $T_\Vvar = U_\Vvar \comp F_\Vvar$ preserves directed colimits and reflexive coinserters.
	This finishes the proof by Corollary~\ref{C:density}.
\end{proof}

\begin{construction}
	\label{C:var}
	Analogously to Construction~\ref{C:eq-q}, to every strongly finitary monad $\Tmon = (T,\mu,\eta)$ on $\CPO$ we assign a variety $\Vvar_\Tmon$ of algebras of the signature $\Sigma_n = |TV_n|$ presented by equations as follows:
	\begin{enumerate}[left= 5pt]
		
		\item $\sigma = \bigvee_{k \in \Nat} \sigma_k$ for every $\omega$-chain $(\sigma_k)_{k < \omega}$ in the cpo $TV_n$ with $\sigma = \bigsqcup_{k \in \Nat} \sigma_k$;
		
		\item $k^*(\sigma) = \sigma(k(x_i))_{i<n}$ for all $\sigma \in \Sigma_n$ and all maps $k : V_n \to \Sigma_m$;
		
		\item $\eta_{V_n}(x_i) = x_i$ for all $i = 0,\dots,n-1$.
		
	\end{enumerate}
\end{construction}

\begin{lemma}
	Every algebra $\alpha : TA \to A$ in $\CPO^\Tmon$ defines an algebra in $\Vvar_\Tmon$ with $\sigma_A(a(x_i)) = a^*(\sigma)$ for all $\sigma \in \Sigma_n$ and $a : V_n \to A$.
	Moreover, every homomorphism in $\CPO^\Tmon$ is also a $\Sigma$-homomorphism between the corresponding algebras in $\Vvar_\Tmon$.
\end{lemma}

The proof is analogous to that of Lemma~\ref{L:homo}.

\begin{theorem}
	\label{T:main-cpo}
	Every strongly finitary monad $\Tmon$ on $\CPO$ is the free-algebra monad of the variety $\Vvar_\Tmon$.
\end{theorem}

\begin{proof}
	By Example~\ref{E:density} the functor $K: \Set_\fin \to \CPO$ has the density presentation consisting of directed colimits and surjective reflexive coinserters.
	Both the free-algebra monad $\Tmon'$ of $\Vvar_\Tmon$ and $\Tmon$ are strongly finitary (Corollary~\ref{C:strong}).
	Therefore, it is sufficient to prove for every finite discrete cpo $X$ that the $\Sigma$-algeba of $\Vvar_\Tmon$ corresponding to the free algebra $(TX,\mu_X)$ of $\CPO^\Tmon$ is a free algebra on $X$ in $\Vvar_\Tmon$.
	In other words, $\Tmon$ and $\Tmon'$ have the same free algebras on objects of $\Set_\fin$.
	Since they both preserve the colimits of the density presentation of $K : \Set_\fin \hookrightarrow \CPO$, it follows that for \emph{every} cpo the free algebras of $\Tmon$ and $\Tmon'$ are the same. Thus the monads $\Tmon$ and $\Tmon'$ are isomorphic, which proves the theorem.
	We can assume $X = V_n$ for some $n \in \Nat$. 
	
	The verification that for every algebra $A$ in $\Vvar_\Tmon$ and every interpretation $f : V_n \to A$, there is a unique $\Sigma$-homomorphism $\ol{f}: T_\Sigma V_n \to A$ with $f = \ol{f} \comp \eta_{V_n}$ is analogous to that in Theorem~\ref{T:mon-var}: In the 'existence' part, $\ol{f}$ is defined by the same formula.
	It is continuous because, given an $\omega$-join $\sigma = \bigsqcup_{k \in \Nat} \sigma_k$ in $TV$, the algebra $A$ satisfies $\sigma = \bigvee_{k \in \Nat} \sigma_k$, thus
	\[\ol{f}(\sigma) = \sigma_A(f(x_i)) = \bigsqcup_{k \in \Nat} (\sigma_k)_A(f(x_i)) = \bigsqcup_{k \in \Nat} f(\sigma_k). \]
	The 'uniqueness' part is identical.
\end{proof}

\begin{corollary}
	Varieties of continuous algebras correspond bijectively, up to isomorphism, to strongly finitary monads on $\CPO$.
\end{corollary}

The proof is analogous to that of Theorem~\ref{T:main-c}.

\section{Varieties of $\Delta$-Continuous Algebras}

We now turn from $\CPO$ to $\DCPO$.

\begin{definition}
	A \emph{$\Delta$-continuous algebra} is a dcpo endowed with continuous operations.
	We denote by $\Sigma\text{-}\DCPO$ the category of $\Delta$-continuous algebras and continuous (directed-joins preserving) homomorphisms.
\end{definition}

We assume again that a signature $\Sigma$ is given, and a countable set $V$ of variables is chosen.

\begin{example}
	A free $\Delta$-continuous algebra on a dcpo $M$ is the algebra $T_\Sigma M$ of terms on variables from $|M|$, see Example~\ref{E:free-c}.
	Indeed, the underlying poset is a coproduct of copies of $|M|^n$ (Example~\ref{E:ex}), thus $T_\Sigma M$ is a dcpo.
	We again use the notation $f^\sharp : T_\Sigma M \to A$ for the homomorphism extending $f : M \to A$.
\end{example}

We are ready to define equations and varieties of $\Delta$-continuous algebras.
The extended terms here have the form $\bigvee_{k < \alpha} t_k$ for arbitrary ordinals $\alpha$.
We namely use the fact that a poset is a dcpo iff it has joins of ordinal-indexed chains, and a map preserves directed joins iff it preserves joins of ordinal-indexed chains (\cite{adamek+rosicky}, Corollary~1.7).
Recall that an ordinal is the linearly ordered set of all smaller ordinals, $\alpha = \{ k \in \Ord \mid k < \alpha \}$.

\begin{definition}
	\phantom{phantom}
	\begin{enumerate}
		
		\item We define \emph{$\Delta$-extended terms} as the smallest set $\T_\Sigma^\Delta V$ containing $T_\Sigma V$ and such that for each ordinal $\alpha$ and every collection $t_k$ ($k < \alpha$) of $\Delta$-extended terms containing only finitely many variables we get a $\Delta$-extended term $\bigvee_{k < \alpha} t_k$.
		
		\item For every $\Delta$-continuous algebra $A$ and every interpretation $f : V \to A$ we denote by $f^@ : \T_\Sigma^\Delta V \rightharpoonup A$ the partial map extending $f^\sharp$ which is defined in $t = \bigvee_{i < \alpha} t_i$ iff each $f^@(t_i)$ is defined and fulfils $f^@(t_i) \sqleq f^@(t_j)$ for all $i < j < \alpha$; then $f^@(t) = \bigsqcup_{i < \alpha} f^@(t_i)$.
		
		\item An \emph{equation} is a pair of extended terms; we write again $t = t'$.
		An algebra $A$ \emph{satisfies} it iff for every interpretation $f : V \to A$ both $f^@(t)$ and $f^@(t')$ are defined, and are equal.
		
	\end{enumerate}
\end{definition}

We thus obtain \emph{varieties} of $\Delta$-continuous algebras as the full subcategories of $\Sigma\text{-}\DCPO$ presented by a set of equations between $\Delta$-extended terms.
Please note that although $\T_\Sigma^\Delta V$ is a proper class, every variety is presented by a \emph{set} (not a proper class) of equations.

\begin{proposition}
	Every variety $\Vvar$ of $\Delta$-continuous algebras has free algebras, and for the ensuing monad $\Tmon_\Vvar$ it is concretely isomorphic to $\DCPO^{\Tmon_\Vvar}$.
\end{proposition}

\begin{proof}
	The existence of free algebras is verified as in Proposition~\ref{P:free}.
	We just need to understand density of a set $X \subseteq C$ for a dcpo $C$ to mean that the only sub-dcpo containing $X$ is all of $C$.
	The rest is, like Proposition~\ref{P:beck}, analogous to the classical case.
\end{proof}

\begin{theorem}
	The monad $\Tmon_\Vvar$ on $\DCPO$ is strongly finitary for every variety of $\Delta$-continuous algebras.
\end{theorem}

The proof is analogous to that of Corollary~\ref{C:strong}.

\begin{construction}
	To every strongly finitary monad $\Tmon$ on $\DCPO$ we assign a variety of $\Delta$-continuous algebras of signature $\Sigma_n = |TV_n|$.
	It is presented by the equations as in Construction~\ref{C:var}, except that in Item 1) we choose arbitrary ordinals
	\[
	\alpha \leq \card |TV|
	\]
	and then form the following equations:
	\[
	\sigma = \bigvee_{k < \alpha} \sigma_k
	\]
	for every $\alpha$-chain $(\sigma_k)_{k < \alpha}$ in the dcpo $TV_n$ with $\sigma = \bigsqcup_{k < \alpha} \sigma_k$.
	Observe that Items 1)-3) yield only a \emph{set} of equations.
\end{construction}

\begin{theorem}
	Every strongly finitary monad on $\DCPO$ is the free-algebra monad of the above variety.
\end{theorem}

The proof is analogous to that of Theorem~\ref{T:mon-var}.

\begin{corollary}
	Varieties of $\Delta$-continuous algebras correspond bijectively, up to isomorphism, to strongly finitary monads on $\DCPO$.
\end{corollary}

\section{Conclusion and Open Problems}

Varieties (aka $1$-basic varieties) of quantitative algebras of Mardare et al.~\cite{mpp16,mpp17} correspond bijectively to strongly finitary monads on the category $\Met$ of metric spaces.
This is the main result of our paper.
It is in surprising contrast to the fact that $\omega$-varieties in op.~cit.\ (where distance restrictions on finitely many variables in equations are considered) do not even yield finitary monads in general, as shown in~\cite{adamek:varieties-quantitative-algebras}.
For varieties of complete quantitative algebras the same result holds: they correspond bijectively to strongly finitary monads on the category $\CMet$ of complete metric spaces.
This relates the quantitative algebraic reasoning of Mardare et al.\ closely to the classical equational reasoning of universal algebra where varieties are known to correspond to finitary monads on $\Set$~\cite{maclane:cwm}.

\begin{openproblem}
	Characterize monads on $\Met$ or $\CMet$ corresponding to $\omega$-varieties of quantitative algebras.
\end{openproblem}

In~\cite{adamek:varieties-quantitative-algebras} a partial answer has been given: when moving from $\Met$ to its full subcategory $\UMet$ on all ultrametric spaces, then enriched monads on $\UMet$ corresponding to $\omega$-varieties of quantitative algebras in $\UMet$ are characterized.

Analogously to the famous Birkhoff Variety Theorem in classical algebra, varieties of quantitative algebras can, as proved in~\cite{mu}, be characterized as precisely the HSP-classes: closed in $\Sigma\text{-}\Met$ under homomorphic images, subalgebras, and products.
However, in $\CMet$ that proof does not seem to work.

\begin{openproblem}
	Does an analogy of the Birkhoff Variety Theorem hold for varieties of complete quantitative algebras?
\end{openproblem}

We have also presented a parallel theory of varieties of continuous algebras.
Here we worked in the category $\CPO$ of $\omega$-cpos (or $\DCPO$ of dcpos), and proved that varieties correspond bijectively to strongly finitary monads on $\CPO$ (or $\DCPO$).
Although the result sounds the same as that for $\Met$, the proof is substantially different.
It relies on $\CPO$ and $\DCPO$ being cartesian closed.

\begin{openproblem}
	Does an analogy of the Birkhoff Variety Theorem hold for varieties of continuous algebras?
\end{openproblem}

In~\cite{anr} an affirmative answer is presented, but the extended terms used there are more general than in our paper.

Our work in $\CPO$ and $\DCPO$ is based on the surprising fact we have proved: in cartesian closed categories directed colimits commute with finite products.

\end{document}